\documentclass[12pt,letterpaper,reqno]{amsart}
\usepackage[utf8]{inputenc}
\usepackage{amsmath}
\usepackage{amsfonts}
\usepackage{amssymb}
\usepackage{amsthm}
\usepackage{bbm}
\usepackage{hyperref}
\usepackage{url}
\usepackage{color}

\newtheorem{thm}{Theorem}
\newtheorem{lem}[thm]{Lemma}

\newtheorem{conj}[thm]{Conjecture}

\numberwithin{thm}{section}
\numberwithin{equation}{section}

\theoremstyle{plain}

\theoremstyle{definition}

\newcommand{\Z}{\mathbb{Z}}

\begin{document}

\title{Special Sets of Primes in Function Fields}

\author{Julio Andrade}
\address{ICERM, Brown University, Providence, RI 02903}
\email{\textcolor{blue}{\href{mailto:julioandrade@brown.edu}{julio\underline{\ }andrade@brown.edu}}}

\author{Steven J. Miller}
\address{Department of Mathematics \& Statistics, Williams College, Williamstown, MA 01267}
\email{\textcolor{blue}{\href{mailto:sjm1@williams.edu}{sjm1@williams.edu}}, \textcolor{blue}{\href{mailto:Steven.Miller.MC.96@aya.yale.edu}{Steven.Miller.MC.96@aya.yale.edu}}}

\author{Kyle Pratt}
\address{Department of Mathematics, Brigham Young University, Provo, UT 84602}
\email{{\textcolor{blue}{\href{mailto:kyle.pratt@byu.net}{kyle.pratt@byu.net}}},{\textcolor{blue}{\href{mailto:kvpratt@gmail.com}{kvpratt@gmail.com}}}}

\author{Minh-Tam Trinh}
\address{Department of Mathematics, Princeton University, Princeton, NJ 08544}
\email{\textcolor{blue}{\href{mailto:mtrinh@princeton.edu}{mtrinh@princeton.edu}}, \textcolor{blue}{\href{mailto:mqtrinh@gmail.com}{mqtrinh@gmail.com}}}



\thanks{The first, third and fourth authors were supported by NSF grant DMS0850577, and the second named author was partially supported by NSF grants DMS0970067 and DMS1265673; the first named author was also partially supported by an ICERM - Brown University Postdoctoral Research Fellowship. We thank Kevin Ford for conversations on the classical case, which led to this research.}

\begin{abstract} When investigating the distribution of the Euler totient function, one encounters sets of primes $\mathcal{P}$ where if $p\in\mathcal{P}$ then $r\in\mathcal{P}$ for all $r|(p-1)$. While it is easy to construct finite sets of such primes, the only infinite set known is the set of all primes. We translate this problem into the function field setting and construct an infinite such set in  $\mathbb{F}_p[x]$ whenever $p \equiv 2$ or $5$ modulo 9.
\end{abstract}

\date{}

\maketitle

\section{Introduction}

Irreducible polynomials over finite fields have applications in cryptography and coding theory \cite{GM94}. Accordingly, it is sometimes useful to be able to produce irreducible polynomials of arbitrarily high degree. Irreducible polynomials over finite fields feature in the Advanced Encryption Standard, and are useful for, among other things, constructing linear feedback shift registers \cite[Chapter 16]{Sch96}. The integer analogy is finding large prime integers to use in some kind of cryptographic application, such as RSA.

Recall that the Euler totient function is given by $\phi(n) = \# \{1 \le k \le n : (k,n)=1\}$. Lehmer's \emph{totient problem} (also sometimes called the \emph{Carmichael conjecture}) asks whether there exists an integer $n_0$ such that $\phi(n)=\phi(n_0)$ implies $n=n_0$. Ford \cite{Ford98} proved that any such $n_0$ must be greater than $10^{10^{10}}$. The general strategy of proof (first initiated by Carmichael \cite{Car22}) is to construct two large set of primes $\mathcal{P}$ and $\mathcal{P}'$ such that every $p$ in $\mathcal{P}$ or $\mathcal{P}'$ is a divisor of a potential $n_0$. If one could show the sets of primes $\mathcal{P}$ and $\mathcal{P}'$ to be infinite, it would follow that no such $n_0$ exists (see \cite{Ford13} for more details). These sets of primes both have the property that if $p$ is in the set and $r|(p-1)$ with $r$ a prime, then $r$ is in the set. Obviously finite sets of this form are easy to construct, and the set of all primes is an infinite set with this property. It is unknown whether there are any nontrivial infinite sets of primes with this property.

Though this question is rather formidable over the integers, it becomes more tractable if we consider the rings of integers of function fields instead. For $q$ a prime power let $\mathbb{F}_q$ be the finite field with $q$ elements. We consider monic, irreducible polynomials in $\mathbb{F}_q[x]$ to be the appropriate analogue of prime numbers, as every monic polynomial can be factored as a product of these polynomials (see \cite{Ros02} for more details). In the function field setting, the relationship between a special infinite set of irreducible polynomials and the Euler totient function is lost, as here the Euler totient function is relatively simple. However, progress towards the problem in the function field setting could shed light on the integer side, and at the very least would indicate the feasibility of such a set $\mathcal{P}$.

We now formulate the central question of study. Let $P$ and $Q$ be monic, irreducible polynomials in $\mathbb{F}_q[x]$ and let $P$ have constant term $\alpha_P$. Consider a set $\mathcal{S} \subset \mathbb{F}_q[x]$ such that if $P \in \mathcal{S}$ and $Q|(P-\alpha_P)$, then $Q \in \mathcal{S}$. Call such a set an $F$-set (think $F$ for \emph{factor}). The requirement that we subtract the constant term of $P$ rather than just 1 is to maintain analogy with the integers. When $p$ is an odd prime, $p-1$ is never a prime since $p-1$ is divisible by 2. Similarly, the polynomial $P-\alpha_P$ is never irreducible, because it is always divisible by $x$.

As with the integer case, finite sets and the set of all monic and irreducible polynomials are easy examples of $F$-sets. On first thought it is unclear whether nontrivial, infinite $F$-sets could exist. The following theorem resolves this question for certain $p$.

\begin{thm}\label{mainTheorem}
Let $p \in \Z$ be a prime such that $p \equiv 2$ or $5$ modulo $9$. Then there is a nontrivial infinite $F$-set in $\mathbb{F}_p[x]$.
\end{thm}

More generally, we believe the following is true.

\begin{conj}
For any finite field $\mathbb{F}_q$, there exists a non-trivial, infinite $F$-set.
\end{conj}

\section{Proof of Main Result}

To begin the proof of Theorem \ref{mainTheorem}, we recall the definition of the order of a polynomial. The order of a polynomial $f$ is the smallest positive integer $e$ such that $f|(x^e-1)$; see Lemma 3.1 of \cite{LN97} for a proof of the existence of the order for every polynomial. We also require the following theorem (we state the theorem in a somewhat weaker form, as this suffices for our purposes).

\begin{thm}[\cite{LN97}, Theorem 3.35]\label{polyConstructThm}
Let $q$ be a prime power. Let $f(x)$ be an irreducible polynomial in $\mathbb{F}_q[x]$ of degree $m$ and order $e$. Let $t$ be a positive integer such that the prime factors of $t$ divide $e$ but not $\frac{q^m-1}{e}$. Assume also that $q^m \equiv 1$ mod 4 if $t \equiv 0$ mod 4. Then $f(x^t)$ is irreducible in $\mathbb{F}_q[x]$.
\end{thm}

We now construct the elements that form our nontrivial infinite $F$-set in $\mathbb{F}_p[x]$. In the course of our arguments we see why we must restrict $p$ to be congruent to $2$ or $5$ modulo 9.

\begin{lem} Let $p \equiv 2$ or $5$ modulo 9, $f_0 = x^2+x+1$ and set $f_\ell(x) = f_0(x^{3^\ell})$. Then $f_\ell$ is a monic, irreducible polynomial in $\mathbb{F}_p[x]$ for all non-negative integers $\ell$. If $g_0(x)=x^2-x+1$ and $g_\ell(x)=g_0(x^{3^\ell})$ then $g_\ell$ is a monic, irreducible polynomial in $\mathbb{F}_p[x]$ for all  non-negative integers $\ell$. \end{lem}

\begin{proof} Suppose $p \equiv 2 \bmod 3$. Then $f_0(x)=x^2+x+1$ is irreducible in $\mathbb{F}_p[x]$. If $f_0(x)$ were reducible, there would be a root $r \in \mathbb{F}_p$ such that \begin{equation} 0 \ \equiv \ r^2 + r + 1 \ \equiv \  (r-1)(r^2+r+1) \ \equiv \ r^3 - 1 \bmod p, \end{equation} with $r \not\equiv 1$ as $3 \not\equiv 0 \bmod p$. Then the order of $r$ in $\mathbb{F}_p^\times$ is 3, but as the order of $\mathbb{F}_p^\times$ is $p-1 \equiv 1 \bmod 3$, such an $r$ is impossible.

The order of $f_0(x)$ is $e=3$, since $f_0(x)|(x^3-1)$. In order for the conditions of Theorem \ref{polyConstructThm} to hold, we need $t=3^\ell$ and $t$ coprime to $\frac{p^2-1}{3}$. This is equivalent to the condition $p^2 \not \equiv 1 \pmod{9}$. Since $p \equiv 2 \bmod 3$, this excludes $p \equiv 8 \bmod 9$. Therefore, by Theorem \ref{polyConstructThm}, the polynomials $f_\ell(x)=f_0(x^{3^\ell})$ are irreducible when $p \equiv 2$ or $5$ modulo $9$ (and are clearly monic).

We similarly see that $g_0(x)=x^2-x+1$ is also a monic, irreducible polynomial in $\mathbb{F}_p[x]$. By Theorem \ref{polyConstructThm} the polynomials $g_\ell(x)=g_0(x^{3^\ell})$ are monic and irreducible since $g_0(x)$ has order 6 and $p \equiv 2,5 \pmod{9}$. \end{proof}

A straightforward induction argument proves the following.

\begin{lem}\label{lem:factorizationsfg}
We have the following factorization:
\begin{eqnarray}
f_\ell(x)-1 & \ = \ & x^{3^\ell}(x+1)g_0(x)g_1(x)\cdots g_{\ell-1}(x) \nonumber\\
g_{\ell}(x)-1 & =  & x^{3^\ell}(x-1)f_0(x)f_1(x)\cdots f_{\ell-1}(x).
\end{eqnarray}
\end{lem}

This lemma shows that $f_\ell(x)-1$ and $g_\ell(x)-1$ have very restricted factorizations, which allows us to prove our main result.

\begin{proof}[Proof of Theorem \ref{mainTheorem}]
Recall that a set $\mathcal{S}$ is an $F$-set if whenever $P\in\mathcal{S}$ and $Q$ then $Q\in\mathcal{S}$ if $Q$ is monic and irreducible and $Q|(P-\alpha_P)$, where $\alpha_P$ is the constant term of $P$. The set \begin{equation} \mathcal{F} \ :=\ \left\{f_{\ell}(x)\right\}_{\ell=0}^{\infty} \ \cup \ \{g_{\ell}(x)\}_{\ell=0}^{\infty} \ \cup \ \{x - n\}_{n=-1}^{1} \end{equation} is an infinite, nontrivial $F$-set. To see this, by Lemma \ref{lem:factorizationsfg} the factors of $f_\ell$ minus its constant term and $g_\ell$ minus its constant term are always in $\mathcal{F}$. All that remains is to check that the factors of the other terms in $\mathcal{F}$ (minus their constant terms) are in $\mathcal{F}$. This follows immediately by construction, as these remaining terms are simply $x$. Note all elements are monic and irreducible.

We must show that $\mathcal{F}$ is nontrivial. To do so, we simply must exhibit \emph{one} monic, irreducible polynomial that is not in $\mathcal{F}$. First consider $p=2$. Then $f_\ell(x)=g_\ell(x)$ for any $\ell$, so $\mathcal{F} = \left\{f_{\ell}(x)\right\}_{\ell=0}^{\infty} \ \cup \ \{x\} \ \cup \ \{x+1\}$. Now note that $x^3+x+1$ is irreducible over $\mathbb{F}_2$, but $x^3+x+1$ is not in $\mathcal{F}$. Thus $\mathcal{F}$ is an infinite, nontrivial $F$-set.

Now let $p > 2$ with $p \equiv 2$ or $5$ modulo $9$. Then $x+2 \in \mathbb{F}_p[x]$ is irreducible but not in $\mathcal{F}$, hence $\mathcal{F}$ is an infinite, nontrivial $F$-set.
\end{proof}



\end{document}